\documentclass[11pt,oneside]{amsart}
\usepackage{amsmath}

\usepackage{amsfonts}

\usepackage{amsthm}

\usepackage{amssymb}

\usepackage[all]{xy}
 \SelectTips{cm}{}
\usepackage{etex}
\usepackage{verbatim,latexsym,exscale,amssymb,amsmath,amsthm,amsfonts,mathrsfs}
\usepackage[all]{xy}
\usepackage[usenames, dvipsnames]{color}
\usepackage{tikz-cd}
\usepackage{relsize}
\usepackage{needspace} 
\usepackage{tikz}
\usepackage{amsmath}
\usepackage{url}
\usepackage{algorithm}
\DeclareMathOperator*{\hocolim}{hocolim}
\DeclareMathOperator*{\holim}{holim}

\newcommand{\rNum}[1]{\lowercase\expandafter{\romannumeral #1\relax}}
\newdir{ >}{{}*!/-5pt/@{>}}
\everymath{\displaystyle}
\usepackage{fontenc}
\usepackage{graphicx}
\usepackage{bm}
\usepackage{listings}
\usepackage{float}
\usepackage{hyperref}
\usepackage{colortbl}
\usepackage{subcaption}
\usepackage{caption}
\usepackage{array}
\usepackage{blkarray}
\usepackage{indentfirst}
\usepackage{inputenc}
\usepackage{booktabs}
\usepackage{dsfont}
\usepackage{algpseudocode,algorithm,algorithmicx}

\definecolor{Gray}{gray}{0.95}
\graphicspath{{Plots/}}

\makeatletter
\newbox\xrat@below
\newbox\xrat@above
\newcommand{\xrightarrowtail}[2][]{%
  \setbox\xrat@below=\hbox{\ensuremath{\scriptstyle #1}}%
  \setbox\xrat@above=\hbox{\ensuremath{\scriptstyle #2}}%
  \pgfmathsetlengthmacro{\xrat@len}{max(\wd\xrat@below,\wd\xrat@above)+.6em}%
  \mathrel{\tikz [>->,baseline=-.75ex]
                 \draw (0,0) -- node[below=-2pt] {\box\xrat@below}
                                node[above=-2pt] {\box\xrat@above}
                       (\xrat@len,0) ;}}
\makeatother

\makeatletter
\newcommand{\xdbheadrightarrow}[2][]{%
  \ext@arrow 0099\xdbheadfill@{#1}{#2}}%
\newcommand{\xdbheadfill@}{%
  \arrowfill@\relbar\relbar{\mathrel{\vphantom{\rightarrow}\smash{\twoheadrightarrow}}}}
\makeatother

 \SelectTips{cm}{}

\title{Rigidity of the $K(1)$-local stable homotopy category}

\author{Jocelyne Ishak}

\email{ji75@kent.ac.uk}

\address{J. Ishak \\ University of Kent \\ School of Mathematics, Statistics and Actruarial Science\\ Sibson building \\ Canterbury, CT2 7FS, UK}

\subjclass[2010]{}

\theoremstyle{theorem}

\newtheorem{thm}{Theorem}[section]
\newtheorem{lemma}[thm]{Lemma}
\newtheorem{thm-defi}[thm]{Theorem-Definition}
\newtheorem{prop}[thm]{Proposition}
\newtheorem{cor}[thm]{Corollary}
\newtheorem*{thm*}{Theorem}
\newtheorem*{rigidity thm*}{Rigidity Theorem}
\newtheorem*{main thm*}{$K(1)$-local Rigidity Theorem}
\newtheorem*{lemma*}{Lemma}
\newtheorem*{cor*}{Corollary}
\newtheorem{exam}[thm]{Example}
\theoremstyle{definition}

\newtheorem{defi}[thm]{Definition}
\newtheorem*{conv*}{Conventions}

\newtheorem*{notation}{Notation}

\newtheorem{rem}[thm]{Remark}
\newtheorem*{remarks*}{Remarks}

\theoremstyle{remark}

\begin{document}

\begin{abstract}

We investigate a new case of rigidity in stable homotopy theory which is the rigidity of the $K(1)$-local stable homotopy category $\mathrm{Ho}(L_{K(1)}\mathrm{Sp})$ at $p=2$. In other words, we show that recovering higher homotopy information by just looking at the triangulated structure of $\mathrm{Ho}(L_{K(1)}\mathrm{Sp})$ is possible, which is a property that only few interesting stable model categories are known to possess.

\end{abstract}

\maketitle
\section*{Introduction}
Model categories were introduced to create a better tool in order to describe homotopy. This enabled us to transform algebraic topology from the study of topological spaces into a wider setting useful in many areas of mathematics. In brief, a model structure on a category $\mathcal{C}$ is a choice of three distinguished classes of morphisms: weak equivalences ($\xrightarrow{\thicksim}$), fibrations ($\twoheadrightarrow$), and cofibrations ($\rightarrowtail$) satisfying certain axioms.

 We can pass to the homotopy category Ho($\mathcal{C}$) associated to a model category $\mathcal{C}$ by inverting the weak equivalences, i.e. by making them into isomorphisms. While the axioms allow us to define the homotopy relations between classes of morphisms in $\mathcal{C}$, the classes of fibrations and cofibrations provide us with a solution to the set-theoretic issues arising in general localisations of categories. Even though it is sometimes sufficient to work in the homotopy category, looking at the homotopy level alone does not provide us with enough higher order structure information. For example, homotopy (co)limits are not usually a homotopy invariant, and in order to define them we need the tools provided by the model category. This is where the question of rigidity may be asked: if we just had the structure of the homotopy category, how much of the underlying model structure can we recover?

This question of rigidity has been investigated during the last decade, and some examples have been studied, but there are still a lot of open questions regarding this fascinating subject. Starting with the stable homotopy category Ho($\mathrm{Sp}$), that is the homotopy category of spectra, Schwede \cite{schwede2007stable} showed that if Ho(Sp) is equivalent as a triangulated category to the homotopy category of a stable model category $\mathcal{C}$, then the model category of spectra is Quillen equivalent to $\mathcal{C}$. In other words, Ho($\mathrm{Sp}$) is \emph{rigid}.

%\begin{rigidity thm*}\cite{schwede2007stable}
%Let $\mathcal{C}$ be a stable model category, and 
%\[
%\begin{array}{ccccc}
%\Phi & : & \mathrm{Ho}(\mathrm{Sp}) & \to & \mathrm{Ho}(\mathcal{C}) \\
%\end{array}
%\]
%an equivalence of triangulated categories. Then the underlying model categories $\mathrm{Sp}$ and $\mathcal{C}$ are Quillen equivalent.
%\end{rigidity thm*}

Now, the next question could be if there is a similar result for Bousfield localisations of the stable homotopy category with respect to certain homology theories. If we look at the part of the stable homotopy category that is readable by a given homology theory, will that structure give us a rigid example? Particularly interesting  localisations are the ones with respect to  Morava $K$-theories $K(n)$  with coefficient ring 
\begin{displaymath}
K(n)_{\ast}\cong \mathbb{F}_p[v_n,v_n^{-1}], \ |v_n|=2p^n-2,
\end{displaymath} 
as well as with respect to Johnson-Wilson theories $E(n)$, with
\begin{displaymath}
E(n)_{\ast}\cong \mathbb{Z}_{(p)}[v_1,v_2,...,v_n,v_n^{-1}], \ |v_i|=2p^i-2.
\end{displaymath}
Both theories at $n=1$ are related to complex $K$-theory in different ways. More precisely, by the Adams splitting, the spectrum $E(1)$ is a summand of complex $K$-theory localised at $p$  \begin{displaymath}
K_{(p)}\cong \bigvee_{i=0}^{p-2}\Sigma^{2i}E(1),
\end{displaymath} 
while $K(1)$ is a summand of mod-$p$ complex $K$-theory \cite[Proposition 1.5.2]{ravenel_orange}. In our case of interest in this article, $p=2$, we have that mod-2 $K$-theory coincides with $K(1)$ since there is only one such summand.

Starting with the Johnson-Wilson theories $E(n)$, for a fixed prime $p$, the localisation of spectra with respect to it is denoted $L_n(\mathrm{Sp})$ (the prime $p$ is omitted from the notation). If we look at the case where $n=1$, and $p=2$, then it has been shown in \cite{roitzheim2007rigidity} that $\mathrm{Ho}(L_1\mathrm{Sp})$ is rigid. However, if we consider the case where $n=1$ and $p\geq 5$, the situation is different since in \cite{franke1996uniqueness}, Franke constructed an exotic algebraic model for the $E(1)$-local stable homotopy category $\mathrm{Ho}(L_1\mathrm{Sp})$ at $p\geq 5$, i.e a model category that realises the same homotopy category but is not Quillen equivalent to $L_1\mathrm{Sp}$. For $p=3$, there is an equivalence, but the question whether it is triangulated remains open, more details about this is discussed in \cite{patchkoria2017exotic}. 

 Now, if we look at $L_n(\mathrm{Sp})$ for other values of $n$ and $p$, little is known about it. For $2p-2>n^2+n$, it has been shown in \cite{franke1996uniqueness} that a potential exotic model exists for $E(n)$-local spectra, although what is known so far is that we have a triangulated equivalence only for $n=1$ and $p \geq 5$. However, for $2p-2\leq n^2+n$, it is still an open question whether we will have rigidity or an exotic model, except the case $n=1$ and $p=2$ which has been shown to be rigid by Roitzheim in \cite{roitzheim2007rigidity}. In particular, for $n=2$ and $p=2$ or $p=3$, the question whether we have rigidity or an exotic model remains unanswered. For further examples of exotics models see \cite{schlichting2002note, dugger2009curious, patchkoria2017derived}, and for other cases of rigidity see \cite{schwede2_local, barnes2014rational, patchkoria2016rigidity, patchkoria2017rigidity}.

Another interesting localisation of spectra that we wish to know more about is the localisation with respect to Morava $K$-theory $K(n)$. In that case, nothing is known about the rigidity  $\mathrm{Ho}(L_{K(n)}\mathrm{Sp})$ or whether we have exotic models. For a fixed prime $p$, $K(n)$-local spectra can be viewed as the difference between $L_n(\mathrm{Sp})$ and $L_{n-1}(\mathrm{Sp})$. More precisely, we have 
\begin{displaymath}
L_n=L_{K(0)\vee K(1)\vee ... \vee K(n)},
\end{displaymath} therefore 
\begin{displaymath}
L_1(\mathrm{Sp})=L_{K(0)\vee K(1)}(\mathrm{Sp}).
\end{displaymath} Since \begin{displaymath}
L_{K(0)}=L_0=L_{H \mathbb{Q}}
\end{displaymath} is rationalisation, the question is whether anything can be said about the rigidity of the $K(1)$-local spectra for $p=2$.

In this article, we investigate one of the open questions mentioned above, which is the rigidity of the $K(1)$-local stable homotopy category $\mathrm{Ho}(L_{K(1)}\mathrm{Sp})$ at $p=2$. While the case of the $E(1)$-local stable homotopy category is related to the $K(1)$-local case, there are a lot of technical differences to keep in mind while studying the $K(1)$-local case. Firstly, unlike the $K(1)$-localisation, the $E(1)$-localisation is smashing, which result in having different compact generators: the local sphere $L_1\mathbb{S}^0$ is a compact generator for the $E(1)$-local case, while a compact generator of $\mathrm{Ho}(L_{K(1)}\mathrm{Sp})$ is given by the  $K(1)$-local mod-2 Moore spectrum $L_{K(1)}M$. Adding to that, 
while $K(1)$-locality implies $E(1)$-locality, the converse is not true i.e. $E(1)$-locality does not imply $K(1)$-locality. Therefore a key theorem used in the proof of the rigidity of $\mathrm{Ho}(L_{E(1)}\mathrm{Sp})$, which is the ``$v_1$-periodicity theorem'' cannot be used in the $K(1)$-local case.
 Our main  result is thus
 \newpage
\begin{main thm*}
Let $\mathcal{C}$ be a stable model category, $p=2$, and let $\Phi$ be an equivalence of triangulated categories 
\[
\begin{array}{ccccc}
\Phi & : & \mathrm{Ho}(L_{K(1)}\mathrm{Sp}) & \to & \mathrm{Ho}(\mathcal{C}). \\
\end{array}
\]
 Then the underlying model categories $L_{K(1)}\mathrm{Sp}$ and $\mathcal{C}$ are Quillen equivalent.
\end{main thm*}
This paper is organised as follows. In the first two sections we recall some definitions surrounding stable model categories and Bousfield localisation. We then start setting up the necessary ingredients to construct the desired Quillen equivalence. The starting point is finding a new characterisation related to $v_1$-self maps to detect $K(1)$-locality. In literature, this is stated for $K$-local spectra, but we can modify it to show that under certain assumptions, it actually proves that a spectrum is $K(1)$-local, which is a stronger statement. After that, we construct a Quillen functor
\begin{displaymath}
L_{K(1)}\mathrm{Sp}\rightarrow \mathcal{C}
\end{displaymath}
by proving that the Quillen functor \begin{displaymath}
\mathrm{Sp}\rightarrow \mathcal{C},
\end{displaymath} 
constructed by the Universal Property of Spectra \cite[5.1]{schwede2002uniqueness}, can be extended to $L_{K(1)}\mathrm{Sp}$ since the right adjoint sends fibrant objects to $K(1)$-local objects. Lastly, we prove that the constructed Quillen functor is a Quillen equivalence by reducing the argument to endomorphisms of the compact generator of $\mathrm{Ho}(L_{K(1)}\mathrm{Sp})$.
%\newpage
\section*{Acknowledgements}
I would like to thank my PhD supervisor Constanze Roitzheim for her professional guidance and valuable support. I am also grateful to John Greenlees and Andy Baker for many helpful discussions. Further thanks go to the referee and Tobias Barthel for valuable comments and suggestions on an earlier version.
\section{Stable model categories}
 We assume that the reader is familiar with basic notions regarding model categories, for example consult \cite{hovey2007model} and \cite{dwyer1995homotopy}.

\begin{defi}
Let $\mathcal{C}$ be a pointed model category and $X \in \mathcal{C}$. First construct $X^{\mathrm{c}}$, a cofibrant replacement of $X$.
One can define a \emph{suspension} of $X$ denoted $\Sigma X$ as the pushout of the diagram 
\begin{displaymath}
\ast \xleftarrow{\ \ \ \  } X^{\mathrm{c}} \amalg X^{\mathrm{c}} \xrightarrowtail[]{\ \ \  }  X^{\mathrm{c}} \wedge I,
\end{displaymath}
where $X^{\mathrm{c}} \wedge I$ is a very good cylinder object of $X^{\mathrm{c}}$.\\
Dually, one can define a \emph{loop object} $\Omega X$ of $X$ by the pullback of the diagram 
\begin{displaymath}
{(X^{\mathrm{f}})}^{I} \xdbheadrightarrow[]{ \ \ \ } {(X^{\mathrm{f}})} \times {(X^{\mathrm{f}})} \longleftarrow \ast ,
\end{displaymath}
where ${(X^{\mathrm{f}})}^{I}$ is a very good path object for a fibrant replacement $X^{\mathrm{f}}$ of $X$.
\end{defi}

%\begin{defi}
%Let $\mathcal{C}$ be a pointed model category and $X \in Ob(\mathcal{C})$. We choose a factorisation $X \rightarrowtail C \xrightarrow {\thicksim} \ast $ of the unique morphism from $X$ into the terminal object. The \emph{suspension} $\Sigma X$ of $X$ is defined as the pushout of the diagram $$ \ast \leftarrow X \rightarrowtail C$$
 %Dually, choosing a factorisation $\ast \xrightarrow {\thicksim} A \twoheadrightarrow X$, the \emph{loop} $\Omega X$ of $X$ is the pullback of the diagram $$ \ast \rightarrow X \twoheadleftarrow A.$$
 %\end{defi}
These constructions are not functorial nor adjoint on $\mathcal{C}$, but they become functorial and adjoint in the homotopy category, and they form an adjunction \begin{displaymath}
 \Sigma :\mathrm{Ho}(\mathcal{C}) \rightleftarrows \mathrm{Ho}(\mathcal{C}):\Omega.
 \end{displaymath}
\begin{notation}
Throughout this paper, we use the following convention: for an adjoint functor pair $F:\mathcal{C} \rightleftarrows \mathcal{D}:G$, the top arrow denotes the left adjoint and the bottom arrow the right adjoint.
\end{notation}

 \begin{defi}
 A pointed model category $\mathcal{C}$ is called \emph{stable} if $\Sigma$ and $\Omega$ are inverse equivalences of homotopy categories.
 \end{defi}
Stable model categories are interesting to study since they carry more structure in their homotopy categories. More precisely, the homotopy category $\mathrm{Ho(\mathcal{C})}$ of a stable model category $\mathcal{C}$ is a triangulated category, where the exact triangles are given by the fiber and cofiber sequences, since in this case they coincide up to sign \cite[7.1.6]{hovey2007model}. Furthermore, Quillen functors between stable model categories induce exact functors on the respective homotopy categories, i.e. functors that respect the triangulated structure. In particular, since the category of spectra $\mathrm{Sp}$ with the  Bousfield-Friedlander model structure \cite{bousfield1978homotopy} is stable, its homotopy category $\mathrm{Ho(Sp)}$ is a triangulated category.
\begin{notation}
We denote the morphisms in a triangulated category $\mathcal{T}$ by $[A,B]^{\mathcal{T}}$. This is a group since triangulated categories are in particular additive. By $[A,B]^{\mathcal{T}}_n$ we mean $[\Sigma ^nA,B]^{\mathcal{T}}$. If $\mathcal{T}=\mathrm{Ho}(\mathcal{C})$ for some stable model category, we write $[A,B]^{\mathcal{C}}$ instead of $[A,B]^{\mathrm{Ho}(\mathcal{C})}$.
\end{notation}

Let $\mathcal{C}$ be a fixed stable model category (for example $\mathcal{C}=\mathrm{Sp}$), and $\mathcal{D}$ any stable model category. Assuming that there is an equivalence of triangulated categories on their homotopy level \begin{displaymath}
\Phi : \mathrm{Ho}(\mathcal{C}) \xrightarrow{\sim} \mathrm{Ho}(\mathcal{D}),
\end{displaymath}
are $\mathcal{C}$ and $\mathcal{D}$ Quillen equivalent?\begin{enumerate}
\item[$\bullet$] If the answer is affirmative, then we say that $\mathrm{Ho}(\mathcal{C})$ is \emph{rigid}. For example, for $\mathcal{C}=\mathrm{Sp}$, Schwede showed that $\mathrm{Ho(Sp)}$ is rigid \cite{schwede2007stable}.
\item[$\bullet$]If the answer is negative, then we have a counterexample where rigidity is not verified, we say that $\mathcal{D}$ is an \emph{exotic model} for  $\mathcal{C}$. 
\end{enumerate}

\begin{defi}
Let $\mathcal{T}$ be a triangulated category with infinite coproducts, and $\mathcal{T'}$ a full triangulated subcategory of $\mathcal{T}$ with shift and triangles induced from $\mathcal{T}$. The subcategory $\mathcal{T'}$ is called \emph{localising} if it is closed under coproducts in $\mathcal{T}$.
\end{defi}
\begin{defi}\label{generator}

 A set $\mathcal{G}$ of objects of a triangulated category  $\mathcal{T}$  is called a set of \emph{generators} if the only localising subcategory containing the objects of $\mathcal{G}$ is $\mathcal{T}$ itself. 
\end{defi} 
 \begin{defi}
  We say that an object $A$ of a triangulated category  $\mathcal{T}$ is \emph{compact} (also called small or finite) if the functor $[A,-]^{\mathcal{T}}$ from $\mathcal{T}$ to groups commutes with arbitrary coproducts, i.e. for any family of objects $\{A_i\}_{i\in \mathcal I}$ whose coproduct exists, the canonical map \[ \textstyle\bigoplus\limits_{i\in I}\left[A,A_i\right]^{\mathcal{T}} \rightarrow \left[A,\coprod A_i\right]^{\mathcal{T}} \] is an isomorphism.  %An object $G$ in a triangulated category $\mathcal{T}$ is a \emph{generator} if it detects isomorphims, meaning that a morphism $X\rightarrow Y$ is an isomorphism if and only if $[G,X]^{\mathcal{T}}\rightarrow [G,Y]^{\mathcal{T}}$ is an isomorphism.
\end{defi}
Note that objects of a stable model category are called ``generators'' or ``compact'' if they are so when considered as objects of the triangulated homotopy category. For a list of interesting examples of compact generators see \cite[Examples 2.3]{schwede2003stable}.\\
The next theorem tells us what criterion should be satisfied by a set of compact objects, in order to become generators of a triangulated category. 

\begin{thm}\label{Keller}\cite[Lemma 2.2.1]{schwede2003stable}
Let $\mathcal{T}$ be a triangulated category with infinite coproducts, and $\mathcal{G}$ a set of compact objects. Then the following are equivalent:
%\begin{samepage}
\begin{itemize}
\item[(i)] The set $\mathcal{G}$ generates $\mathcal{T}$ in the sense of Definition \ref{generator}.
\item[(ii)]The objects of $\mathcal{G}$ detect isomorphisms, meaning that a morphism \mbox{$X\rightarrow Y$} in $\mathcal{T}$ is an isomorphism if and only if $[G,X]^{\mathcal{T}}\rightarrow [G,Y]^{\mathcal{T}}$ is an isomorphism for all $G \in \mathcal{G}$. 
\end{itemize}
%\end{samepage}
\end{thm}
The previous theorem will consist an important step in the proof of the main result in this article. Briefly speaking, if we want to prove that a criterion is true for all the objects in a certain triangulated category, then it is often sufficient to prove it true for a compact generator.
\begin{rem}\label{rem_generator}
Note that in Theorem \ref{Keller}, the point (i) implies (ii) even without the hypotheses of compactness. In other words, the objects of a set of generators detect isomorphisms.
\end{rem}
\section{Chromatic homotopy theory} 
\begin{defi}
 A \emph{spectrum} $X$ is a sequence of pointed simplicial sets ${\displaystyle(X_0,X_1,\cdots)}$ together with structure maps \begin{displaymath}
 \sigma_n^X: \Sigma X_n \rightarrow X_{n+1}, \  \text{or equivalently}
 \end{displaymath}\begin{displaymath}
 \bar{\sigma}^X_n: X_n \rightarrow \Omega X_{n+1}.
 \end{displaymath} A morphism $f:X \rightarrow Y$ of spectra is a collection of morphisms of pointed sets \begin{displaymath}
 f_n:X_n\rightarrow Y_n
\end{displaymath} that commute with the structure maps, that is, \begin{displaymath}
f_{n+1}\circ \sigma_n^X=\sigma_n^Y\circ \Sigma f_n, \  \text{ for all } n\geq 0.
\end{displaymath}A spectrum $X$ is called a \emph{suspension spectrum} (respectively an \emph{$\Omega$-spectrum}) if $\sigma_n^X$ (respectively $\bar{\sigma}^X_n$) is a weak homotopy equivalence for all $n$.
 
 \end{defi}
 \begin{notation}
Throughout this paper, $\mathrm{Sp}$ denotes the category of spectra with the stable Bousfield-Friedlander model structure \cite{bousfield1978homotopy}. The mod-$p$ Moore spectrum is denoted $M(\mathbb{Z}/p)$, and is the cone of multiplication by $p$ on the sphere spectrum, i.e. it is part of a distinguished triangle in $\mathrm{Ho(Sp)}$
\begin{displaymath}
 \mathbb{S}^0\xrightarrow{.p} \mathbb{S}^0\xrightarrow{\mathrm{incl}} M(\mathbb{Z}/p)\xrightarrow{\mathrm{pinch}} \Sigma\mathbb{S}^0.
\end{displaymath}
Here, incl is the inclusion of the bottom cell, and pinch is the map that ``pinches'' off the bottom cell so that only the top cell is left.
 \end{notation}
 
 Bousfield localisation restricts attention to the part of the stable homotopy theory visible to a given homology theory $E_{\ast}$, which makes this tool very useful in studying the stable homotopy category. The main references for such constructions are \cite{bousf} and \cite{ravenel1984localization}. This construction becomes particularly interesting when looking at some very special homology theories that give information about the structure of the $p$-local stable homotopy catgeory for some prime $p$. In our case, we are interested in localisation with respect to Morava $K$-theory $K(1)$ at $p=2$, with the following model structure.

Let $\mathrm{Sp}$ be the category of spectra with the Bousfield-Friedlander model structure. Then there is a model category $L_{K(1)} \mathrm{Sp}$ with the same objects as $\mathrm{Sp}$ and with the following model structure.
\begin{enumerate}
\item[$\bullet$]The weak equivalences are the $K(1)_{\ast}$-equivalences.
\item[$\bullet$]The cofibrations are the cofibrations of $\mathrm{Sp}$.
\item[$\bullet$]The fibrations are the maps with the right lifting property with respect to cofibrations that are $K(1)_{\ast}$-equivalences.

\end{enumerate}

\begin{rem}
The model categories $\mathrm{Sp}$ and $L_{K(1)}\mathrm{Sp}$ have the same cofibrant objects, but the fibrant objects in $L_{K(1)}\mathrm{Sp}$ are the one which are fibrant in $\mathrm{Sp}$ and $K(1)_{\ast}$-local. The set of homotopy classes of maps in $\mathrm{Ho}(L_{K(1)}\mathrm{Sp})$ is denoted \begin{displaymath}
[X,Y]^{L_{K(1)}\mathrm{Sp}}=[L_{K(1)}X,L_{K(1)}Y]^{\mathrm{Sp}} \cong [X,L_{K(1)}Y]^{\mathrm{Sp}}.
\end{displaymath} 
\end{rem} 
\begin{defi}
Localisation at $E$ is said to be \emph{smashing} if for every spectrum $X$, the map \begin{displaymath}
\mathrm{Id} \wedge \eta_{\mathbb{S}^0}:X \rightarrow X\wedge L_E\mathbb{S}^0
\end{displaymath} 
is an $E$-localisation.
\end{defi}
A nice feature of the functors $L_n$ is that they are \emph{smashing} unlike the functors $L_{K(n)}$.
\begin{thm}[Smash product theorem] \cite[Theorem 7.5.6]{ravenel_orange}
For any spectrum $X$, \begin{displaymath}
L_nX \simeq X \wedge L_n\mathbb{S}^0.
\end{displaymath}

\end{thm}

\begin{rem}\label{smashing}
The $E$-localisation functor $L_E$ is triangulated and preserves generators. However, it does not preserve compactness in general, because the functor $L_E$ does not commute with arbitrary coproducts. However, if the localisation is smashing then the functor $L_E$ commutes with arbitrary coproducts \cite[Proposition 1.27(d)]{ravenel1984localization}, and the image of a compact generator is again a compact generator. Therefore, the spectrum $L_E\mathbb{S}^0$ is a generator in $\mathrm{Ho}(L_E\mathrm{Sp})$, but it is compact for a smashing localisation like $L_{E(n)}$. 
\end{rem}

On the other hand, localisation with respect to the $n^{th}$-Morava $K$-theory is not smashing for $n>0$. Although the $K(n)$-local sphere is still a generator, it is not a compact one. However, the following result provides us with a compact generator for $\mathrm{Ho}(L_{K(1)}\mathrm{Sp})$.
\begin{lemma}\cite[Theorem 7.3]{hovey_strickland_morava} \emergencystretch 3em{For a fixed prime $p$, the spectrum $L_{K(1)}M(\mathbb{Z}/p)$ is a compact generator for the $K(1)$-local stable homotopy category  $\mathrm{Ho}(L_{K(1)}\mathrm{Sp})$.}
\end{lemma}
\begin{defi}
The localisation of a spectrum $X$ with respect to the mod-$p$ Moore spectrum $M(\mathbb{Z}/p)$ is the \emph{$p$-completion} of $X$ denoted $X_p^{\wedge}$, i.e.
\begin{displaymath}
X_p^{\wedge}=L_{M(\mathbb{Z}/p)}X.
\end{displaymath}
If a spectrum is $M(\mathbb{Z}/p)$-local, then we call it a $p$\emph{-complete spectrum.} 
\end{defi}
\begin{prop}\cite[Proposition 2.5]{bousf}\label{mod-p completion}
\begin{itemize}
\item[$\mathrm{(a)}$]For $X\in \mathrm{Ho(Sp)}$, its $p$-completion is the function spectrum:\begin{displaymath}
X_p^{\wedge}=L_{M(\mathbb{Z}/p)}X\simeq F(\Omega M(\mathbb{Z}/p^{\infty}),X),
\end{displaymath}
where $\mathbb{Z}/p^{\infty}$ can be defined as the factor group $\mathbb{Z}[1/p]/ \mathbb{Z}$, or as the colimit of the groups $\mathbb{Z}/p^n$ under multiplication by $p$. Additionally, there is a split short exact sequence 
\begin{displaymath}
 0 \xrightarrow{}\mathrm{Ext}(\mathbb{Z}/p^{\infty},\pi_{\ast}X) \xrightarrow{} \pi_{\ast}(L_{M(\mathbb{Z}/p)}X) \xrightarrow{} \mathrm{Hom}(\mathbb{Z}/p^{\infty},\pi_{\ast -1}X) \xrightarrow{} 0.
 \end{displaymath}
\item[$\mathrm{(b)}$]If the groups $\pi_{\ast}X$ are finitely generated, then  $$\pi_{\ast}( L_{M(\mathbb{Z}/p)}X )\cong \mathbb{Z}^{\wedge}_p\otimes \pi_{\ast}X,$$ where $\mathbb{Z}^{\wedge}_p$ denotes the $p$-adic integers.
\item[$\mathrm{(c)}$]A spectrum $X\in \mathrm{Ho(Sp)}$ is $M(\mathbb{Z}/p)$-local (equivalently p-complete) if and only if the groups $\pi_{\ast}(X)$ are $\mathrm{Ext}$-$p$-complete in the following sense: 
\begin{itemize}
\item[(\rNum{1})] The completion map \begin{displaymath}
\pi_{\ast}(X) \rightarrow \mathrm{Ext}(\mathbb{Z}/p^{\infty}, \pi_{\ast}(X) )
\end{displaymath} is an isomorphism, and
\item[(\rNum{2})] $\mathrm{Hom}(\mathbb{Z}/p^{\infty}, \pi_{\ast}(X))=\ast. $

\end{itemize}
\end{itemize}

\end{prop}
Moreover, $p$-completion can be described as a homotopy limit.
\begin{cor}\label{cor_completion}
As a consequence of Proposition \ref{mod-p completion}, the $p$-completion of a spectrum $X$ is 
$$\displaystyle X^{\wedge}_p\simeq \holim(... \xrightarrow{} M(\mathbb{Z}/p^3)\wedge X \xrightarrow{} M(\mathbb{Z}/p^2)\wedge X\xrightarrow{}M(\mathbb{Z}/p)\wedge X). $$
\end{cor}

Another feature of the $K(1)$-localisation that will become useful in the next section is the following result, which enables us to see the $K(1)$-localisation as the $p$-completion of the $E(1)$-localisation.
  % another reference is \cite[Propostion 7.10.(e)]{hovey_strickland_morava}
\begin{lemma} \cite[Proposition 2.11]{bousf} \label{L_K(1)}
For a fixed prime $p$ and $X$ any spectrum in $\mathrm{Ho(Sp)}$, we have 
\begin{displaymath}
L_{K(1)}X=L_{M(\mathbb{Z}/p)}L_1X=(L_1X)^{\wedge}_p.
\end{displaymath}
\end{lemma}
We will end this section by talking about the ``$v_1$-self map'' that will be needed later on. 

\begin{defi}
Let $X$ be a $p$-local finite spectrum, and let $n\geq 1$. A $v_n$-\emph{self map} is a map $f:\Sigma ^k X \rightarrow X$ with the following properties:
\begin{enumerate}
\item[(a)] The map $f$ is a $K(n)_{\ast}$-equivalence.
\item[(b)] For $m\neq n$, the induced map $K(m)_{\ast}(X)\rightarrow K(m)_{\ast}(Y) $ is nilpotent.
\end{enumerate}
\end{defi}
\begin{defi}
We say that a $p$-local finite spectrum $X$ has \emph{type n} if  \begin{displaymath}
K(n)_{\ast}(X)\neq 0, \ \text{but} \ K(m)_{\ast}(X)= 0 \ \  \text{for} \ m< n.
\end{displaymath}
\end{defi}
\begin{exam}
A spectrum $X$ has type 0 if \begin{displaymath}
H_{\ast}(X, \mathbb{Q}) \ncong 0,
\end{displaymath}
or equivalently if $H_{i}(X, \mathbb{Z})$ is not a torsion group for all $i$. An example of such a spectrum is the $p$-local sphere $\mathbb{S}^0_{(p)}$.

\end{exam}
\begin{exam}
An example of a spectrum of type 1 is the mod-p Moore spectrum $M(\mathbb{Z}/p)$. To begin with, it has no rational homology  
\begin{displaymath}
K(0)_{\ast}(M(\mathbb{Z}/p))=H_{\ast}(M(\mathbb{Z}/p), \mathbb{Q})=0.
\end{displaymath}
Furthermore, the non-triviality of $K(1)_{\ast}(M(\mathbb{Z}/p))$ can be deduced by considering the cofiber sequence \begin{displaymath}
 \mathbb{S}^0 \xrightarrow{.p} \mathbb{S}^0 \rightarrow M(\mathbb{Z}/p).
 \end{displaymath}
 More precisely, the maps \begin{displaymath}
K(1)_{\ast}(\mathbb{S}^0)\rightarrow K(1)_{\ast}(M(\mathbb{Z}/p))
\end{displaymath}are injections, because multiplication by $p$ kills $K(1)_{\ast}(\mathbb{S}^0)\cong \mathbb{F}_p[v_1,v_1^{-1}]$, hence $K(1)_{\ast}(M(\mathbb{Z}/p))$ cannot be trivial.
 %Moreover, if we look at the cofiber sequence \begin{displaymath}
 %\mathbb{S}^0 \xrightarrow{.p} \mathbb{S}^0 \rightarrow M(\mathbb{Z}/p),
 %\end{displaymath}
%the associated long exact sequence in $K(1)$-homology is of the form 
%\begin{displaymath}
%...\rightarrow  K(1)_{i+1}M(\mathbb{Z}/p)\rightarrow K(1)_{i}\mathbb{S}^0 \xrightarrow{.p}K(1)_{i}\mathbb{S}^0\rightarrow  K(1)_{i}M(\mathbb{Z}/p)\rightarrow %...
%\end{displaymath}
%Since multiplication by $p$ kills $K(1)_{\ast}(\mathbb{S}^0)\cong \mathbb{F}_p[v_1,v_1^{-1}]$, the maps \begin{displaymath}
%%K(1)_{\ast}(\mathbb{S}^0)\rightarrow K(1)_{\ast}(M(\mathbb{Z}/p))
%\end{displaymath} 
%are injections, and $K(1)_{\ast}(M(\mathbb{Z}/p))$ is non-trivial.
\end{exam}
\begin{thm}[Periodicity Theorem]\cite[Chapter 6]{ravenel_orange} \cite[ \textsection 3]{hopkins_smith}
Let $X$ be a finite $p$-local spectrum of type $n$. Then $X$ admits a $v_n$-self map
\begin{displaymath}
v_n^{p^i}: \Sigma ^{p^id}X \rightarrow X, \text{ for some } \ i\geq 0.
\end{displaymath} 
Where $d=0$ if $n=0$, and $d=2p^n-2$ if  $n>0$.

\end{thm}
By applying this theorem to the mod-2 Moore spectrum, we get the following $v_1$-self map.
\begin{exam}\label{Adams map}
The earliest known periodic map was constructed by \cite{ADAMS_J(X)4}, now known as the Adams map. It is denoted 
\begin{displaymath}
v_1^4:\Sigma^8M(\mathbb{Z}/2)\rightarrow M(\mathbb{Z}/2).
\end{displaymath}
Note that there is no smaller degree $v_1$-self map that can be realised by $M(\mathbb{Z}/2)$. 
\end{exam}

\section{From E(1)-locality to K(1)-locality}\label{E(1) and K(1)}
\begin{notation}
From now on, let $p=2$, and let $\mathrm{Sp}$ denote the category of 2-local spectra. The mod-2 Moore spectrum $M(\mathbb{Z}/2)$ will be denoted by $M$. 
\end{notation}

In \cite{bousf}, a criterion involving the Adams periodic map $v_1^4$ has been developed to show that a spectrum is $E(1)$-local: 
\begin{lemma}\cite[\S 4]{bousf}
A spectrum $X$ is $E(1)$-local if and only if $v_1^4$ induces an isomorphism
\begin{displaymath}
{(v_1^4)}^{\ast}:[M,X]_n^{\mathrm{Sp}}\rightarrow [M,X]_{n+8}^{\mathrm{Sp}}, \ \text{for all} \ n\in \mathbb{Z}.
\end{displaymath}

\end{lemma}

In this section we extend this result to $K(1)$-locality by adding another condition. First, we need the following lemma.
%\\The mod-2 Moore spectrum $M$ is a torsion spectrum with $K(1)_{\ast}(M)$ being non-trivial. Hence, by \cite[Section 3]{hopkins_smith} it has a ``$v_1$-self map'', i.e. a map $v_1^4:\Sigma^8M\rightarrow M$ that induces an isomorphism in $K(1)_{\ast}$. (The notation might seem misleading at first: $v_1^4$ is not actually the fourth power of an existing map $v_1$. It alludes to the fact that the smallest such map has degree 8, whereas the element $v_1\in K(1)_{\ast}$ has degree 2).
%We can use this to test whether a spectrum is $K(1)$-local.

\begin{lemma}\label{lemma_rem}
For any spectrum $X \in \mathrm{Ho(Sp)}$, we have 
\begin{displaymath}
 L_{K(1)}(M\wedge X)\simeq L_1(M\wedge X).
\end{displaymath}
\end{lemma}
\begin{proof}
By \cite[3.9]{dwyer2004localizations}, we have the following homotopy pullback square
\[
\begin{tikzcd}[column sep=large]
 L_1Y \arrow{r}{} \arrow{d}[left]{}
&{L_{K(1)}Y} \arrow{d} \\
 L_0Y \arrow{r}[swap]{} & L_0L_{K(1)}Y.  
\end{tikzcd}
\]

Therefore, we have that if $L_0Y\simeq\ast$ and $L_{K(1)}Y\simeq\ast$ then $L_1Y \simeq L_{K(1)}Y$. This is the case for 
\begin{displaymath}
Y=X\wedge M(\mathbb{Z}/2):=X/2.
\end{displaymath}
First, let us prove that $L_0(X\wedge M)\simeq\ast$. The long exact homotopy sequence of the exact triangle $$\displaystyle X \xrightarrow{.2} X \xrightarrow{incl} X/2\xrightarrow{}\Sigma X $$ splits into short exact sequences of the form $$\displaystyle 0 \xrightarrow{} (\pi_{m+1}X)\big /2 \xrightarrow{} \pi_{m+1}(X/2)\xrightarrow{} \{\pi_mX\}_2 \rightarrow 0, $$Here $\{\pi_mX\}_2$ denotes the $2$-torsion of the group $\pi_mX$. Since tensoring with $\mathbb{Q}$ preserves exactness, we have 
\begin{displaymath}
\pi_{m+1}(X/2)\otimes \mathbb{Q}\cong 0 \cong \pi_{m+1}(L_0(X/2)),
\end{displaymath}
therefore $L_0(X/2)\simeq\ast$. \\The same applies to $L_0L_{K(1)}(X/2)$. We can see that by tensoring the following short exact sequence with $\mathbb{Q}$ $$\displaystyle 0 \xrightarrow{} (\pi_{m+1}L_{K(1)}X)\big /2 \xrightarrow{} \pi_{m+1}L_{K(1)}(X/2)\xrightarrow{} \{\pi_mL_{K(1)}X\}_2 \rightarrow 0, $$ we will have that $L_0L_{K(1)}(X/2)\simeq\ast$. Hence, $ L_1(M\wedge X)\simeq L_{K(1)}(M\wedge X)$ as desired.
\end{proof}

\begin{rem}
Even though the above lemma is written in the 2-local world, we can replace $p=2$ in the proof by any prime $p$ and the lemma will still be correct in the $p$-local setting.

\end{rem}
\begin{lemma}\label{Ishak_lemma}
A 2-complete spectrum $X$ is $K(1)$-local if and only if $v_1^4$ induces an isomorphism $${(v_1^4)}^{\ast}:[M,X]_n^{Sp}\rightarrow [M,X]_{n+8}^{Sp}$$ for all $n\in \mathbb{Z}$.
\end{lemma} 
\begin{proof}
First, suppose that the spectrum $X$ is $K(1)$-local. As we have seen, the map $v_1^4$ induces a $K(1)_{\ast}$-isomorphism on $M$, thus its cofibre $V(1)$ is $K(1)_{\ast}$-acyclic. The desired isomorphism is deduced from the long exact sequence 
$$\displaystyle ... \xrightarrow{} [V(1),X]_{n+1} \xrightarrow{}[M,X]_{n} \xrightarrow{{(v_1^4)}^{\ast}} [\Sigma ^ 8 M,X]_{n} \rightarrow [V(1),X]_{n} \rightarrow ..., $$ 
since by hypothesis $[V(1),X]_{n}=0$ for all $n$ .

To prove the other direction, we first note that the assumption is equivalent to $$(v_1^4\wedge X)^{\ast}:\pi_n(M\wedge X)\rightarrow \pi_{n+8}(M\wedge X)$$ being  an isomorphism for all $n$ because $M$ is its own Spanier-Whitehead dual up to suspension  $DM(\mathbb{Z}/p)\simeq \Omega M(\mathbb{Z}/p)$. We conclude that $$\hocolim ( \displaystyle M\wedge X\xrightarrow{v_1^4\wedge X} \Sigma^{-8}M \wedge X \xrightarrow{v_1^4\wedge X}\Sigma^{-16}M \wedge X\xrightarrow{}...) \simeq M\wedge X$$ because all the arrows are weak equivalences. However, by \cite[Proposition 4.2]{bousf} $L_1M$ is the homotopy colimit of the sequence formed by the self-map on $M(\mathbb{Z}/2)$, i.e. we have that $$\hocolim ( \displaystyle M\xrightarrow{v_1^4} \Sigma^{-8}M \xrightarrow{v_1^4}\Sigma^{-16}M\xrightarrow{v_1^4}...) \simeq L_1M.$$
 We conclude that in our case, $$M\wedge X\simeq (L_1M)\wedge X$$
since unlike $\holim$, $\hocolim$ commutes with the smash product ``$\wedge$''. Therefore, $$M\wedge X \simeq (L_1M)\wedge X \simeq L_1(M\wedge X),$$ because $L_1$ is smashing. On the other hand, by Lemma \ref{lemma_rem} 
\begin{displaymath}
L_1(M\wedge X)\simeq L_{K(1)}(M\wedge X) .
\end{displaymath}
  We conclude that $M\wedge X $ is $K(1)$-local.\\By induction, we prove that $M(\mathbb{Z}/2^n)\wedge X$ is $K(1)$-local for all $n$. The octahedral axiom provides us with the following exact triangle in $\mathrm{Ho(Sp)}$
  \begin{displaymath}
   M(\mathbb{Z}/2)\wedge X \rightarrow M(\mathbb{Z}/2^n)\wedge X \rightarrow M(\mathbb{Z}/2^{n-1})\wedge X \rightarrow \Sigma M(\mathbb{Z}/2)\wedge X.
  \end{displaymath}
  If we suppose that $M(\mathbb{Z}/2^{n-1}) \wedge X$ is $K(1)$-local, then \cite[Lemma 1.4]{bousf} tells us that $M(\mathbb{Z}/2^{n}) \wedge X$ is $K(1)$-local as well. By Corollary \ref{cor_completion}, the 2-completion of $X$ denoted $X^{\wedge}_2$ is the homotopy limit of$$ \displaystyle ...\xrightarrow{} M(\mathbb{Z}/2^n) \wedge X \xrightarrow{}...\xrightarrow{} M(\mathbb{Z}/2^2) \wedge X \xrightarrow{}M(\mathbb{Z}/2)\wedge X  .$$
 Since every term of the above sequence is $K(1)$-local, the spectrum  $X^{\wedge}_2$ is $K(1)$-local \cite[Proposition 1.7]{ravenel1984localization}. As $X$ is 2-complete, this must mean that $X$ itself is $K(1)$-local.
\end{proof}
\section{The Quillen functor pair}
In order to obtain a Quillen equivalence between $L_{K(1)}\mathrm{Sp}$ and $\mathcal{C}$, we first need a Quillen adjunction between those categories. Forgetting the $K(1)$-local structure, Quillen adjunctions between spectra $\mathrm{Sp}$ and any stable model category have been studied first in \cite{schwede2002uniqueness} and were later generalised in \cite{lenhardt2012stable}.
\begin{thm}\cite[Section 6]{lenhardt2012stable}\\Let $\mathcal{C}$ be a stable model category and $X\in \mathcal{C}$ a fibrant and cofibrant object. Then there is a Quillen adjunction $$X\wedge -:\mathrm{Sp}\rightleftarrows \mathcal{C}:\mathrm{Hom}(X,-)$$such that $X\wedge \mathbb{S}^0\simeq X$.
\end{thm}
\begin{notation}
 The left derived functor of $X\wedge -:\mathrm{Sp}\rightarrow \mathcal{C}$ is denoted \mbox{$X\wedge ^L - :\mathrm{Ho(Sp)}\rightarrow \mathrm{Ho}(\mathcal{C})$}, and $\mathrm{RHom}(X,-):\mathrm{Ho}(\mathcal{C})\rightarrow\mathrm{Ho(Sp)} $ denotes the right derived functor of $\mathrm{Hom}(X,-)$.
\end{notation}
Looking at $\mathrm{Sp}$ and $L_{K(1)}\mathrm{Sp}$ as categories, they are the same, however they have different model structures. We would like to show that this construction also gets us a Quillen adjunction between $K(1)$-local spectra $L_{K(1)}\mathrm{Sp}$ as follows 
\[
\begin{tikzcd}[column sep=large]
\mathrm{Sp} \arrow{r}{X\wedge -} \arrow{d}[left]{\mathrm{Id}}
&{\mathcal{C}} \\
 L_{K(1)}\mathrm{Sp}. \arrow{ur}[swap]{X\wedge -} & 
\end{tikzcd}
\]

By \cite[Proposition 7.8]{barnes_roitzheim_local}, this is the case if and only if the spectrum $\mathrm{RHom}(X,Y)$ is $K(1)$-local for all $Y\in \mathcal{C}$.

For the rest of the paper, let $\Phi  :  \mathrm{Ho}(L_{K(1)}\mathrm{Sp}) \rightarrow \mathrm{Ho}(\mathcal{C})$ be an equivalence of triangulated categories, and $X$ a fibrant-cofibrant replacement of $\Phi (L_{K(1)}\mathbb{S}^0)$. In order to show that  $\mathrm{RHom}(X,Y)$ is $K(1)$-local for all $Y$, we use Lemma \ref{Ishak_lemma}. However, before being able to apply Lemma \ref{Ishak_lemma}, we need to prove that $\mathrm{RHom}(X,Y)$ is $2$-complete, to that end we use Proposition \ref{mod-p completion}.
%\begin{prop}
%A spectrum $X \in \mathrm{Ho(Sp)}$ is 2-complete if and only if the groups $\pi_{\ast}X$ are Ext-2-complete.
%\end{prop}
\begin{lemma}
The spectrum $\mathrm{RHom}(X,Y)$ is 2-complete  for all $Y\in \mathcal{C}$.
\end{lemma}
\begin{proof}
By Proposition \ref{mod-p completion} (c), in order to prove that the spectrum $\mathrm{RHom}(X,Y)$ is 2-complete  for all $Y\in \mathcal{C}$, it is enough to show that the groups $\pi_{\ast}(\mathrm{RHom}(X,Y))$ are Ext-$2$-complete. However, the latter fact is the same as the following equivalent statements: 
\begin{itemize}
%\item[] The groups $[\mathbb{S}^0, \mathrm{RHom}(X,Y)]^{\mathrm{Sp}}$ are Ext-2-complete
\item[(\romannumeral 1)] The groups $[X\wedge ^{L}\mathbb{S}^0,Y]_{\ast}^{\mathcal{C}}$ are Ext-2-complete, because the pair ($X\wedge ^L -,\mathrm{RHom}(X,-)$) is an adjunction.
 \item[(\romannumeral 2)] The groups $[X,Y]_{\ast}^{\mathcal{C}}$ are Ext-2-complete since $X\wedge^L\mathbb{S}^0\cong X$.
 \item[(\romannumeral 3)] The groups $[\Phi ^{-1}(X),\Phi ^{-1}(Y)]_{\ast}^{L_{K(1)}\mathrm{Sp}}$ are Ext-2-complete ($\Phi$ is an equivalence of categories).
  \item[(\romannumeral 4)] The groups $[\mathbb{S}^0,\Phi ^{-1}(Y)]_{\ast}^{L_{K(1)}\mathrm{Sp}}$ are Ext-2-complete because $\Phi ^{-1}(X) \cong L_{K(1)}\mathbb{S}^0$.
  \item[(\romannumeral 5)] The groups $[\mathbb{S}^0,L_{K(1)} \Phi ^{-1}(Y)]_{\ast}^{\mathrm{Sp}}$ are Ext-2-complete, as a consequence of the isomorphism $[X,Y]^{L_{K(1)}\mathrm{Sp}} \cong [X,L_{K(1)}Y]^{\mathrm{Sp}}$.\end{itemize}
The last statement is the same as saying that the spectrum $L_{K(1)} \Phi ^{-1}(Y)$ is 2-complete, which is indeed true, since by Lemma \ref{L_K(1)} we have 
\begin{displaymath}
L_{K(1)}\Phi ^{-1}(Y)=\Big(L_1 \Phi ^{-1}(Y)\Big)^{\wedge}_2. \qedhere
\end{displaymath}  
  
\end{proof}

%We can use the above proposition to show that $\mathrm{RHom}(X,Y)$ is $2$-complete  for all $Y\in \mathcal{C}$. Actually, the fact that $\mathrm{RHom}(X,Y)$ is $2$-complete is equivalent to saying that the groups $[X\wedge ^{L}\mathbb{S}^0,Y]_{\ast}^{\mathcal{C}}$ are Ext-2-complete, because of the adjunction. By keeping in mind the facts that $X\wedge^L\mathbb{S}^0\simeq X$ and that $\Phi$ is an equivalence of categories, we see that in order to show that $\mathrm{RHom}(X,Y)$ is 2-complete, it is sufficient to show that $[\mathbb{S}^0,L_{K(1)}\Phi^{-1}(Y)]_{\ast}^{Sp}$ are Ext-2-complete, in other words that is equivalent to saying that $L_{K(1)}\Phi^{-1}(Y)$ is 2-complete, which is indeed true, since $L_{K(1)}\Phi^{-1}(Y)\simeq \left(L_1 \Phi^{-1}(Y)\right)^{\wedge}_2$. Actually, the $K(1)$-localisation is the $p$-completion of the $E(1)$-localisation. Therefore, we conclude that $\mathrm{RHom}(X,Y)$ is 2-complete.
\begin{lemma} \label{Image of M}
For the mod-2 Moore spectrum $M$, we have \begin{displaymath}
X\wedge^L M \cong \Phi (L_{K(1)}M) \cong  \Phi (M).
\end{displaymath}
%as both are the cofibre of the multiplication by 2 on the element $X$. 
\end{lemma}
\begin{proof}
From the isomorphisms \begin{displaymath}
X\wedge^L \mathbb{S}^0 \cong X   \cong  \Phi (L_{K(1)}\mathbb{S}^0),
\end{displaymath}
we can see that both $X\wedge^L M$ and $\Phi (L_{K(1)}M)$ are the cofibre of the multiplication by 2 on the element $X$, hence we have 
\begin{displaymath}
X\wedge^L M\cong \Phi (L_{K(1)}M).
\end{displaymath}
As for the last isomorphism of the lemma, it is a consequence of the isomorphism  $$M\cong L_{K(1)}M$$ in the homotopy category $\mathrm{Ho}(L_{K(1)} \mathrm{Sp})$.
\end{proof}

In order to show that $\mathrm{RHom}(X,Y)$ is $K(1)$-local for all $Y$, we use Lemma \ref{Ishak_lemma}.

\begin{lemma}\label{correction1}
The map $$ (v_1^4)^{\ast}:[M,\mathrm{RHom}(X,Y)]^{\mathrm{Sp}}_n\rightarrow [M,\mathrm{RHom}(X,Y)]^{\mathrm{Sp}}_{n+8}$$ is an isomorphism for all $n \in \mathbb{Z}$ and all $Y \in \mathcal{C}$.
\end{lemma}
\begin{rem} \label{sphere} 
Before we proceed to the proof of the above lemma, we need to know the homotopy groups of the $K(1)$-local sphere in degrees 0 till 9. By \cite [Theorem 4.3]{bousf} and Lemma \ref{L_K(1)}, we can see that the $K(1)$-local sphere is the fiber of the Adams operation $\Psi^3 - 1$ on $KO\mathbb{Z}_2$, where $KO\mathbb{Z}_2$ is the 2-adic real $K$-theory spectrum. Therefore the long exact sequence produced by the fiber sequence 
\begin{displaymath}
L_{K(1)}\mathbb{S}^0 \rightarrow KO\mathbb{Z}_2 \xrightarrow{\Psi^3 - 1} KO\mathbb{Z}_2
\end{displaymath} 
 provides us with values of $\pi_n(L_{K(1)}\mathbb{S}^0)$ at $p=2$. On the other hand, the long exact sequence provided by the homotopy pullback square \[
\begin{tikzcd}[column sep=large]
 L_1Y \arrow{r}{} \arrow{d}[left]{}
&{L_{K(1)}Y} \arrow{d} \\
 L_0Y \arrow{r}[swap]{} & L_0L_{K(1)}Y 
\end{tikzcd}
\]
tells us that 2-locally, we have \begin{displaymath}
\pi_n(L_{K(1)}\mathbb{S}^0) \cong \pi_n(L_{1} \mathbb{S}^0), \ \text{for} \ n\neq -2,\  -1, 0.
\end{displaymath} 
\pagebreak
The final result from degree -2 till 9 reads as follows, see e.g. \cite[Theorem 8.15]{ravenel1984localization}. 
\begin{center}
\begin{table}[!h]
           %% not "\fontsize{12}{12}\selectfont"
    \caption{}\label{table}
    \centering    %% not "\center{...}"
    \begin{tabular}{|c|c|}
    \hline
    $n$&$\pi_n(L_{K(1)}\mathbb{S}^0)$\\     %% no "&" at start of row
    \hline
    $-2$& $0$ \\
    \hline
     $-1$& $\mathbb{Z}_{2}^{\wedge}$ \\
    \hline
    $0$& $\mathbb{Z}_{2}^{\wedge} \{ \iota \} \oplus \mathbb{Z}/2 \{y_0 \}$ \\
    \hline
    $1$& $\mathbb{Z}/2 \{\eta, y_1 \}$ \\
    \hline   
    $2$& $\mathbb{Z}/2 \{\eta ^2 \}$ \\
    \hline 
    $3$& $\mathbb{Z}/8 \{ \nu \}$ \\
    \hline 
    $4$& $0$ \\
    \hline 
    $5$& $0$ \\
    \hline 
    $6$& $0$ \\
    \hline 
    $7$& $\mathbb{Z}/16 \{ \sigma \}$ \\
    \hline 
    $8$& $\mathbb{Z}/2 \{ \eta \sigma \}$ \\
    \hline 
    $9$& $\mathbb{Z}/2 \{ \eta ^2 \sigma, \mu \}$ \\
    \hline 
         %% extra \hline at bottom of table
    \end{tabular}
  \end{table}
\end{center}

The element $y_0$ is the unique element of order 2 of $\pi_0(L_{K(1)}\mathbb{S}^0)$, and $y_1=\eta y_0$ is a generator of the second summand in $\pi_1(L_{K(1)}\mathbb{S}^0)$. As for the other elements of $\pi_n(L_{K(1)}\mathbb{S}^0)$, we give them the names of their preimages in $\pi_n(\mathbb{S}^0)$. Moreover, we have the following relations, (\cite[Theorem 8.15(d)]{ravenel1984localization} 
\begin{displaymath}
4\nu = \eta ^3, \ \eta y_1=0, \ y_0 ^2= 0, \ y_1^2=0, \ \sigma y_1=0 \ \  \text{and} \ \  \mu y_0= \eta^2 \sigma.
\end{displaymath}

\end{rem}
Now, we can move on to proving Lemma \ref{correction1}, namely that the mod-2 homotopy groups of $\mathrm{RHom}(X,Y)$ are $v_1^4$-periodic for all $Y\in \mathcal{C}$.
\begin{proof}
By adjunction, it suffices to prove that $$(X\wedge^L v_1^4)^{\ast}:[X\wedge^LM,Y]_n^{\mathcal{C}} \rightarrow [X\wedge^LM,Y]_{n+8}^{\mathcal{C}}$$ is an isomorphism for all integers $n$. We know that \begin{displaymath}
{(v_1^4)}^{\ast} : [M,\Phi ^ {-1}(Y)]_n^{L_{K(1)}\mathrm{Sp}} \rightarrow [M,\Phi ^ {-1}(Y)]_{n+8}^{L_{K(1)}\mathrm{Sp}}
\end{displaymath}
is an isomorphism for all $n$,  but by Lemma \ref{Image of M}, this means that
\begin{displaymath}
\Phi(v_1^4) ^{\ast}: [X\wedge^LM,Y]_n^{\mathcal{C}}\rightarrow [X\wedge^LM,Y]_{n+8}^{\mathcal{C}}
\end{displaymath}
is an isomorphism as well. Therefore, to show that $(X\wedge^L v_1^4)^{\ast}$ is an isomorphism, one compares the elements $(X\wedge^L v_1^4)$ and $\Phi (v_1^4)$ in the endomorphism ring $$[X\wedge ^LM,X\wedge ^LM]_8^{\mathcal{C}} \cong [M,M]_8^{L_{K(1)}\mathrm{Sp}}.$$
But since $[M,M]_{\ast}^{L_{K(1)}\mathrm{Sp}}=[M,M]_{\ast}^{L_1\mathrm{Sp}}$ by Lemma \ref{lemma_rem}, we can use the calculation done in \cite[Section 3.2]{roitzheim2007rigidity} for the $E(1)$-local case to show that \begin{displaymath}
X\wedge^L v_1^4= \Phi (v_1^4) + \Phi (T),\ \text{for some} \ T\in [M,M]_{8}^{L_{K(1)}\mathrm{Sp}}, \ 2T=0.
\end{displaymath} 
However, by Remark \ref{sphere} and \cite[Section 3.2]{roitzheim2007rigidity} we can see that all such $v_1^4+T$ are isomorphims in $\mathrm{Ho}(L_{K(1)}\mathrm{Sp})$. Hence, $(X\wedge^L v_1^4)^{\ast}$ is also an isomorphism. 
\end{proof}
Since we have seen that $\mathrm{RHom}(X,Y)$ is 2-complete, we can make use of the last lemma with Lemma \ref{Ishak_lemma} to deduce the following:
\begin{cor}\label{K(1)-local}
The spectrum $\mathrm{RHom} (X,Y)$ is $K(1)$-local for all $Y$. Thus, $$X\wedge -: L_{K(1)}\mathrm{Sp} \rightleftarrows \mathcal{C}:\mathrm{Hom}(X,-)$$ is a Quillen adjunction. \qed
\end{cor}

\section{The Quillen equivalence} \label{equivalence}
As before, let  $\Phi  :  \mathrm{Ho}(L_{K(1)}\mathrm{Sp}) \rightarrow \mathrm{Ho}(\mathcal{C})$ be an equivalence of triangulated categories. After constructing the Quillen adjunction $$X\wedge-:L_{K(1)}\mathrm{Sp}\rightleftarrows \mathcal{C}:\mathrm{Hom}(X,-)$$ where $X\simeq \Phi (L_{K(1)}\mathbb{S}^0)$, our goal now is to prove that this Quillen adjunction is indeed a Quillen equivalence. To this end, we first start by looking at the homotopy type of the spectrum $\mathrm{RHom}(X,X\wedge^LM)$. Note that in the $E(1)$-local case in \cite{roitzheim2007rigidity}, the author investigated the homotopy type of $\mathrm{RHom}(X,X\wedge^L\mathbb{S}^0)$. The reason behind it is that in $L_1\mathrm{Sp}$, the sphere spectrum is a compact generator, while in $L_{K(1)}\mathrm{Sp}$, the Moore spectrum $M$ is a compact generator, and $\mathbb{S}^0$ is just a generator. Everything mentioned and the reason why we are looking at a compact generator will become apparent when we will be proving the equivalence.\\As we have seen in Corollary \ref{K(1)-local}, the spectrum $\mathrm{RHom}(X,X\wedge^LM)$ is $K(1)$-local. Therefore, by the universal property of localisation \cite[Proposition 1.5]{ravenel1984localization}, the adjoint of the identity map factors over $L_{K(1)}M$
\[
\begin{tikzcd}
M \arrow{r}{} \arrow{d}{\mathrm{Id}}
&{\mathrm{RHom}(X,X\wedge^LM)} \\
 L_{K(1)}M \arrow[dotted]{ur}[below]{\lambda} & 
\end{tikzcd}
\]
\begin{prop}\label{isomorphism}
The map $\lambda:L_{K(1)}M\rightarrow \mathrm{RHom}(X,X\wedge^LM)$ is a $\pi_{\ast}$-isomorphism.
\end{prop}

\begin{proof}
As all the homotopy groups involved are torsion, it is enough to show that $\lambda$ induces an isomorphism of mod-$2$ homotopy groups. In other words, we need to show that $\lambda_{\ast}$ in the following commutative diagram is an isomorphism.

\[
\begin{tikzcd}
{[M,L_{K(1)}M]_{\ast}^{\mathrm{Sp}}=[M,M]_{\ast}^{L_{K(1)}\mathrm{Sp}}} \arrow{r}{\lambda_{\ast}} \arrow{d}[left]{X\wedge^L-}
&{[M,\mathrm{RHom}(X,X\wedge^LM)]_{\ast}^{\mathrm{Sp}}} \\
 {[X\wedge^LM,X\wedge^LM]_{\ast}^{\mathcal{C}}} \arrow{ur}[below]{\mathrm{adj}}[above]{\cong} & 
\end{tikzcd}
\]
It is commutative because, by definition of $\lambda$, for $\alpha \in [M,L_{K(1)}M]_{\ast}^{\mathrm{Sp}}$, the image of $X\wedge ^L \alpha$ under the adjunction isomorphism is precisely $\lambda \circ \alpha$. All we need to show is that $$X\wedge^L-: [M,M]_n^{L_{K(1)}\mathrm{Sp}}\longrightarrow [X\wedge ^LM, X\wedge^LM]_n^{\mathcal{C}}$$ is an isomorphism for all $n$.\\ However, via the self map $v_1^4$, the endomorphisms of the Moore spectrum are periodic of period 8 in $\mathrm{Ho}(L_{K(1)} \mathrm{Sp})$
\begin{displaymath}
[M,M]_n^{L_{K(1)}\mathrm{Sp}}\cong [M,M]_{n+8}^{L_{K(1)}\mathrm{Sp}}.\end{displaymath}

 Therefore, we only have to show that the desired isomorphism holds for $n=1,...,8$. To that end, we show that $$X\wedge^L-: [\mathbb{S}^0,\mathbb{S}^0]_n^{L_{K(1)}\mathrm{Sp}}\longrightarrow [X,X]_n^{\mathcal{C}}$$ 
is an isomorphism for $n=0,...,9$ by verifying that 
\begin{displaymath}
\psi : [\mathbb{S}^0,\mathbb{S}^0]_n^{L_{K(1)}\mathrm{Sp}} \xrightarrow{X\wedge^L-}[X,X]_n^{\mathcal{C}}\xrightarrow{\Phi^{-1}} [\mathbb{S}^0,\mathbb{S}^0]_n^{L_{K(1)}\mathrm{Sp}}
\end{displaymath}
is an isomorphism in that range. By Remark \ref{sphere} and \cite[Lemma 3.4]{roitzheim2007rigidity}, we deduce that the desired isomorphism is established for $n=1,...,9$ because the homotopy groups of $L_1\mathbb{S}^0$ and $L_{K(1)}\mathbb{S}^0$ agree in degrees 1 to 9. As for the degree 0, the situation is similar to \cite{roitzheim2007rigidity} because $y_0$ is still the only nonzero torsion element in $\pi_0(L_{K(1)}\mathbb{S}^0)=\mathbb{Z}_{2}^{\wedge} \{ \iota \} \oplus \mathbb{Z}/2 \{y_0 \} $, hence we have that \begin{displaymath}
 X\wedge^L y_0 = \Phi (y_0), \ \text{by \cite[Lemma 3.4]{roitzheim2007rigidity}}.
\end{displaymath} Therefore, the morphism $\Psi$ is an isomorphism at degree 0 as well. The desired result will follow by using a five lemma argument. To be more specific, we have the following commutative diagram \[
\begin{tikzcd}
0\arrow{r}&
{(\pi_nL_{K(1)}\mathbb{S}^0)\big /2}\arrow{r}{\mathrm{incl}_{\ast}}\arrow{d}
&{\pi_n(L_{K(1)}M)}\arrow{r}{\mathrm{pinch}_{\ast}}\arrow{d}{}
&{\{ \pi_{n-1}L_{K(1)}\mathbb{S}^0\}_2}\arrow{r}{}\arrow{d}{}
&
0\\
0\arrow{r}&
{[X,X]_n^{\mathcal{C}}\big / 2}\arrow{r}{\mathrm{incl}_{\ast}}&{[X,X\wedge M]_n^{\mathcal{C}}}\arrow{r}{\mathrm{pinch}_{\ast}}& {\{[X,X]_{n-1}^{\mathcal{C}} \}_{2}} \arrow{r}{} &0
\end{tikzcd}
\]
where the two rows are short exact sequences, and the left hand side, as well as the right hand side, are isomorphisms for $n=0,...,9$. Therefore, we conclude that the middle row is an isomorphism. Now, the statement that $$X\wedge^L-: [M,M]_n^{L_{K(1)}\mathrm{Sp}}\longrightarrow [X\wedge ^LM, X\wedge^LM]_n^{\mathcal{C}}$$ is an isomorphism for $n=1,...,8$ is deduced from the following commutative diagram \[
\begin{tikzcd}[column sep=1.2em]
0\arrow{r}&
{(\pi_{n+1}(L_{K(1)}M))\big /2}\arrow{r}{\mathrm{pinch}^{\ast}}\arrow{d}{\cong}
&{[M,M]_n^{L_{K(1)}\mathrm{Sp}}}\arrow{r}{\mathrm{incl}^{\ast}}\arrow{d}{}
&{\{ \pi_{n}L_{K(1)}M\}_2}\arrow{r}{}\arrow{d}{\cong}
&0\\
0\arrow{r}&{[X\wedge^L \mathbb{S}^0,X\wedge^L M]_{n+1}^{\mathcal{C}}\big /2}\arrow{r}{}&{[X\wedge^L M,X\wedge^L M]_n^{\mathcal{C}}}\arrow{r}{}& {\{[X\wedge^L \mathbb{S}^0,X\wedge^L M]_{n}^{\mathcal{C}} \}_{2}} \arrow{r}{} &0.
\end{tikzcd}
\]
Thus, we can conclude that $L_{K(1)}M$ and $\mathrm{Hom}(X,X\wedge^LM)$ are weakly equivalent in $\mathrm{Sp}$.

\end{proof}
Now that we have all the necessary arguments, we can use the fact that $M$ is a compact generator of $\mathrm{Ho}(L_{K(1)}\mathrm{Sp})$ to prove our main theorem.
\begin{thm}
The Quillen adjunction $$X\wedge -:L_{K(1)}\mathrm{Sp}\rightleftarrows\mathcal{C}:\mathrm{Hom}(X,-)$$ is a Quillen equivalence.
\end{thm}
\begin{proof}
By \cite[1.3.16]{hovey2007model}, it is sufficient to show the following:
\begin{enumerate}
\item[$\bullet$]$\mathrm{RHom}(X,-): \mathrm{Ho}(\mathcal{C})\rightarrow \mathrm{Ho}(L_{K(1)}\mathrm{Sp})$ reflects isomorphisms.
\item[$\bullet$]$A\rightarrow \mathrm{RHom}(X,X\wedge^LA)$ is an isomorphism for all $A\in \mathrm{Ho}(L_{K(1)}\mathrm{Sp})$.
\end{enumerate}
Since $\Phi$ is an equivalence of triangulated categories, $\Phi (L_{K(1)}\mathbb{S}^0)=X$ is a generator for $\mathrm{Ho}(\mathcal{C})$, therefore as mentioned in the first section in Remark \ref{rem_generator} it detects isomorphisms.\\ Let us first show the first point. For a morphism $f:Y\rightarrow Z$ in $\mathcal{C}$, let \begin{displaymath}
\mathrm{RHom}(X,f): \mathrm{RHom}(X,Y)\rightarrow \mathrm{RHom}(X,Z)
\end{displaymath} be an isomorphism in $\mathrm{Ho}(L_{K(1)}\mathrm{Sp})$, so $$[\mathbb{S}^0, \mathrm{RHom}(X,Y)]_{{}_{\ast}}^{L_{K(1)}(\mathrm{Sp})}\xrightarrow{\mathrm{RHom}(X,f)}[\mathbb{S}^0, \mathrm{RHom}(X,Z)]_{{}_{\ast}}^{L_{K(1)}(\mathrm{Sp})}$$ is an isomorphism. By adjunction, $$[X,Y]^{\mathcal{C}}_{\ast}\xrightarrow{f_{\ast}} [X,Z]_{\ast}^{\mathcal{C}}$$ is an isomorphism. Since $X$ is a generator in $\mathrm{Ho}(\mathcal{C})$, we have that $$f:Y\rightarrow Z$$ is an isomorphism in $\mathrm{Ho}(\mathcal{C})$ which proves the first point.\\In order to prove the second point, we use Theorem \ref{Keller} mentioned in the first section of this paper. Consider the full subcategory $\mathcal{T}$ of $\mathrm{Ho}(L_{K(1)}\mathrm{Sp})$ containing those $A\in \mathrm{Ho}(L_{K(1)}\mathrm{Sp})$ such that 
\begin{displaymath}
A\rightarrow \mathrm{RHom}(X,X\wedge ^L A)
\end{displaymath} is an isomorphism. Our goal is to prove that $\mathcal{T}= \mathrm{Ho}(L_{K(1)}\mathrm{Sp})$. Since $\mathrm{RHom}(X,-)$ and $X\wedge ^L -$ are exact functors, $\mathcal{T}$ is triangulated. By Proposition \ref{isomorphism} it contains the Moore spectrum $M$, i.e. a compact generator of $\mathrm{Ho}(L_{K(1)}\mathrm{Sp})$.\\ 
In order to use Theorem \ref{Keller}, we still need to verify that this category $\mathcal{T}$ is also closed under coproducts. Now let $A_i$, $i\in \mathcal{I}$, be a family of objects in $\mathcal{T}$. We would like that $\textstyle\coprod\limits_{i} A_i \in \mathcal{T}$. As $M$ reflects isomorphisms, this means that we need to show that $$[M,\textstyle\coprod\limits_{i} A_i]_{{}_{\ast}}^{L_{K(1)}\mathrm{Sp}}\rightarrow [M,\mathrm{RHom}(X,X\wedge ^L(\textstyle\coprod\limits_{i} A_i))]_{{}_{\ast}}^{L_{K(1)}\mathrm{Sp}} $$ is an isomorphism. On the other hand, by adjunction, $$[M,\mathrm{RHom}(X,X\wedge ^L(\textstyle\coprod\limits_{i} A_i))]_{{}_{\ast}}^{L_{K(1)}\mathrm{Sp}} \cong [X\wedge ^LM, X\wedge ^L (\textstyle\coprod\limits_{i} A_i)]_{{}_{\ast}}^{\mathcal{C}} .$$ Since $X\wedge ^L-$ is a left adjoint, it commutes with coproducts, therefore $$[X\wedge ^LM, X\wedge ^L (\textstyle\coprod\limits_{i} A_i)]_{\ast}^{\mathcal{C}}\cong [X\wedge ^LM,\textstyle\coprod\limits_{i} (X\wedge ^L A_i)] _{\ast}^{\mathcal{C}}.$$ Since $\Phi$ is an equivalence of triangulated categories, and $L_{K(1)}M$ is a compact generator of $\mathrm{Ho}(L_{K(1)}\mathrm{Sp})$, we have that $\Phi(L_{K(1)}M)=X\wedge ^LM$ is a compact generator of $\mathrm{Ho}(\mathcal{C})$, and this means that$$ [X\wedge ^LM,\textstyle\coprod\limits_{i} (X\wedge ^L A_i)] _{\ast}^{\mathcal{C}}\cong \textstyle\bigoplus\limits_{i} [X\wedge^LM, X\wedge ^LA_i]_{\ast}^{\mathcal{C}} .$$Similarly, we know that $$[M,\textstyle\coprod\limits_{i} A_i]_{{}_{\ast}}^{L_{K(1)}\mathrm{Sp}}\cong  \textstyle\bigoplus\limits_{i} [M,A_i]_{{}_{\ast}}^{L_{K(1)}\mathrm{Sp}} .$$ As $A_i \in \mathcal{T}$, for all $i\in \mathcal{I}$, $$[M, A_i] _{{}_{\ast}}^{L_{K(1)}\mathrm{Sp}}\cong [M, \mathrm{RHom}(X,X\wedge ^L A_i)] _{{}_{\ast}}^{L_{K(1)}\mathrm{Sp}},$$ which is induced by $$A_i\xrightarrow{\cong} \mathrm{RHom}(X,X\wedge ^L A_i).$$ By naturality of those isomorphisms, we have that $\mathcal{T}$ is closed under coproducts, therefore $\mathcal{T}= \mathrm{Ho}(L_{K(1)}\mathrm{Sp})$, and our Quillen adjunction is indeed a Quillen equivalence.
\end{proof}
What is known so far is that, 2-locally, we have rigidity of Ho($L_1$Sp) and Ho($L_{K(1)}$Sp). It will be intriguing to find out if at odd primes, an exotic model for $L_{K(1)}$Sp can be constructed out of the one already established for the $E(1)$-local case.
%In conclusion,  are rigid for $p=2$. However, for $p$ odd, an exotic model has been constructed for $L_1$Sp, it will be intriguing to find out if that case can be linked with $L_{K(1)}$Sp at odd primes.
%The rigidity of the $K(1)$-local stable homotopy category at $p=2$ is now proved. It will be intriguing to find out if we can construct an exotic model for Ho($L_K(1)$Sp) at $p$ odd, and try to link it to the exotic model of $L_1$Sp for $p$ odd.

\bibliographystyle{alpha}

\end{document}